\documentclass[a4paper,11pt]{article}

\usepackage[utf8]{inputenc}
\usepackage[UKenglish]{babel}
\usepackage[T1]{fontenc}
\usepackage{amsmath,amsfonts,amssymb,amsthm}
\usepackage{graphicx}
\usepackage[top=2cm,bottom=3cm,left=2.5cm,right=2.5cm]{geometry}
\usepackage{setspace}
\usepackage{xcolor}
\usepackage{stmaryrd}
\usepackage{enumitem}
\usepackage{esint}
\usepackage{centernot}
\usepackage[format=hang,singlelinecheck=false]{caption}
\usepackage{framed}
\usepackage[normalem]{ulem}

\usepackage[final,colorlinks=true]{hyperref}

\author{Guillaume Blanc\thanks{École Polytechnique Fédérale de Lausanne, \href{mailto:guillaume.blanc@epfl.ch}{guillaume.blanc@epfl.ch}; \url{https://sites.google.com/view/guillaume-blanc-math}} \and Alice Contat\thanks{Université Sorbonne Paris Nord, \href{mailto:alice.contat@math.cnrs.fr}{alice.contat@math.cnrs.fr}; \url{https://sites.google.com/view/acontat/}}}
\title{Blow-up rate of solution to generalised Blasius equation}
\date{}

\begin{document}

\theoremstyle{plain}
\newtheorem{thm}{Theorem}
\newtheorem{prop}{Proposition}
\newtheorem{lem}{Lemma}
\newtheorem{cor}{Corollary}
\newtheorem{claim}{Claim}

\theoremstyle{remark}
\newtheorem{rem}{Remark}

\maketitle

\begin{abstract}
We identify the blow-up rate of a solution to a generalised Blasius equation, that we came across while studying a probabilistic model of ``Poissonian burning'' in Euclidean space.
Our proof involves the study of the long-time behaviour of solutions to a Lotka--Volterra system.
\end{abstract}

\section*{Introduction and main results}

In this paper, we consider the equation
\begin{equation}\label{eq:cauchy}
\begin{cases}
y^{(d+1)}(t)=y(t)\cdot y^{(d)}(t),\quad t\in\mathbb{R}_+,\\\
\text{$y(0),\dots,y^{(d-1)}(0)=0$ and $y^{(d)}(0)=1$,}
\end{cases}
\end{equation}
where $d\in\mathbb{N}^*$ is a fixed parameter.
By the Cauchy--Lipschitz theorem, the Cauchy problem \eqref{eq:cauchy} admits a unique maximal solution $y:{[0,T[}\rightarrow\mathbb{R}$, where $T>0$, and we have either $T=\infty$ (the solution is global), or $T<\infty$ and $y(t)\rightarrow\infty$ as $t\to T^-$ (the solution blows up in finite time).
In fact, for $d=1$, the maximal solution of \eqref{eq:cauchy} is given explicitly by
\[y:t\in\left[0,\frac{\pi}{\sqrt{2}}\right[\longmapsto\sqrt{2}\cdot\tan\left(\frac{t}{\sqrt{2}}\right).\]
In particular, this solution blows up in finite time $T=\left.\pi\middle/\sqrt{2}\right.$, and its blow-up rate is given by
\[y(t)\sim2\cdot\frac{1}{T-t}\quad\text{as $t\to T^-$.}\]
For $d=2$, Equation \eqref{eq:cauchy} comes up in combinatorics, in counting unordered trilabelled increasing trees \cite{kubapanholzer}.
As noticed by Kuba and Panholzer, the maximal solution $y:{[0,T[}\rightarrow\mathbb{R}$ of~\eqref{eq:cauchy} can be related to the solution of the Blasius equation from boundary layer theory:
\begin{equation}\label{eq:blasius}
\begin{cases}
z'''(t)+z(t)\cdot z''(t)=0,&t\in\mathbb{R}_+,\\
\text{$z(0),z'(0)=0$ and $\lim_{t\to\infty}z'(t)=1$.}&
\end{cases}
\end{equation}
The Blasius equation has been treated by several authors in the mathematics literature, including Weyl \cite{weyl}, who first rigorously proved the existence of a solution to \eqref{eq:blasius} by considering the Cauchy problem
\begin{equation}\label{eq:weyl}
\begin{cases}
z'''(t)+z(t)\cdot z''(t)=0,\quad t\in\mathbb{R}_+,\\\
\text{$z(0),z'(0)=0$ and $z''(0)=1$.}
\end{cases}
\end{equation}
In particular, he gave bounds on the radius of convergence of the power series expansion at $t=0$ of the maximal solution to \eqref{eq:weyl}, which in fact corresponds to the explosion time $T$ of the maximal solution $y$ of \eqref{eq:cauchy}.
Interestingly, Weyl also treated the following generalisation of \eqref{eq:weyl}, which is almost \eqref{eq:cauchy}:
\[\begin{cases}
z^{(d+1)}(t)+z(t)\cdot z^{(d)}(t)=0,&t\in\mathbb{R}_+,\\
\text{$z(0),\ldots,z^{(d-1)}(0)=0$ and $z^{(d)}(0)=1$,}
\end{cases}\]
where $d\in\mathbb{N}^*$ is a parameter.
In this paper, we investigate the blow-up of the maximal solution $y:{[0,T[}\rightarrow\mathbb{R}$ to \eqref{eq:cauchy}, for any $d\in\mathbb{N}^*$.

\paragraph{Blow-up rate of $y$.}
Our main results are stated in Theorem \ref{thm:main} below.
To the best of our knowledge, except for the case $d=1$ which is explicitly solvable, such results were only available for $d=2$, and the corresponding proofs do not straightforwardly extend to $d\geq3$.
They were obtained originally by Coppel \cite{coppel}, and revisited more recently by Ishimura and Matsui \cite{ishimuramatsui}.
\begin{thm}\label{thm:main}
Fix $d\in\mathbb{N}^*$, and let $y:{[0,T[}\rightarrow\mathbb{R}$ be the maximal solution of \eqref{eq:cauchy}.
This solution blows up in finite time: we have $T<\infty$, and $y(t)\rightarrow\infty$ as $t\to T^-$.
Moreover, we identify its blow-up rate in the following sense:
\begin{itemize}
\item For $d\in\llbracket1,10\rrbracket$, the following holds: for each $i\in\llbracket0,d\rrbracket$, we have
\begin{equation}\label{eq:goalstrong} y^{(i)}(t)\sim(d+1)\cdot\frac{i!}{(T-t)^{i+1}}\quad\text{as $t\to T^-$.} \end{equation}

\item Regardless of the value of $d$, the following holds: for each $i\in\llbracket0,d\rrbracket$, we have
\[y^{(i)}(t)=\Theta\left(\frac{1}{(T-t)^{i+1}}\right)\quad\text{as $t\to T^-$,}\]
and most importantly, we have
\begin{equation}\label{eq:goal}
\int_0^ty(s)\mathrm{d}s\sim(d+1)\cdot\ln\left(\frac{1}{T-t}\right)\quad\text{as $t\to T^-$.}
\end{equation}
\end{itemize}
\end{thm}

Of importance to us is \eqref{eq:goal}, because of its consequence for a probabilistic model of ``Poissonian burning'' in Euclidean space, as we explain below.
Of course, if \eqref{eq:goalstrong} holds then it implies \eqref{eq:goal}; but contrary to \eqref{eq:goal}, it might not be the case that \eqref{eq:goalstrong} holds for all values of $d$.
We will comment more on this once we give an overview of the proof of Theorem \ref{thm:main}.

\paragraph{Probabilistic motivation.}
We came across this problem while studying the following probabilistic model of ``Poissonian burning'' in Euclidean space.
Consider a Poisson process $\Pi$ with intensity $\mathrm{d}x\otimes f(t)\mathrm{d}t$ on $\mathbb{R}^d\times[0,T[$, where $f:{[0,T[}\rightarrow\mathbb{R}_+$ is a function to be adjusted later.
For each atom $(x,t)$ of $\Pi$, think of $t\in[0,T[$ as the time at which the point $x\in\mathbb{R}^d$ is set on fire, and picture the fire as then propagating at speed $1/2$ in $\mathbb{R}^d$, see Figure \ref{fig:burning1d}. 

\begin{figure}[ht]
\begin{center}
\includegraphics[width=15cm]{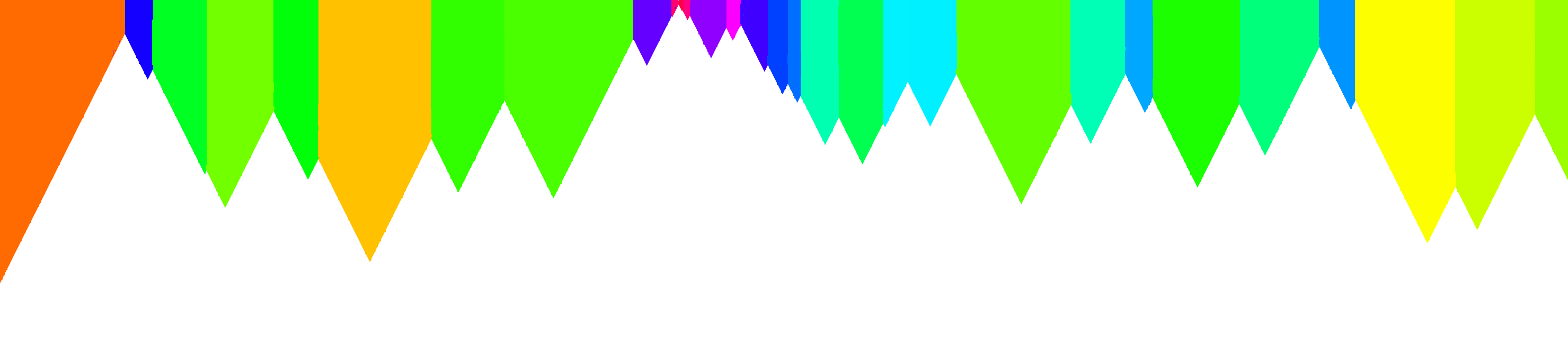}
\caption{Simulation of the model of Poissonian burning for $d=1$, in the window $[-5,5]$. Space is depicted horizontally, whereas time evolves vertically. At each time, so for each horizontal layer, we color the points that are burned with the same color as the first point of the process that set them on fire.}\label{fig:burning1d}
\end{center}
\end{figure}

Motivated by a discrete version of the burning model on the Euclidean lattice $\mathbb{Z}^d$, we equip $\mathbb{R}^d$ with the $1$-norm ${\|\cdot\|_1:x=(x_1,\ldots,x_d)\in\mathbb{R}^d\mapsto|x_1|+\ldots+|x_d|}$.
The set of points that are burned by time $t\in[0,T[$ is
\[\mathcal{B}(t)=\bigcup_{\substack{(x,s)\in\Pi\\s\leq t}}\overline{B}_{\|\cdot\|_1}\left(x,\frac{t-s}{2}\right),\]
where $\overline{B}_{\|\cdot\|_1}(x,r)=\left\{z\in\mathbb{R}^d:\|x-z\|_1\leq r\right\}$ denotes the $1$-norm ball with centre $x\in\mathbb{R}^d$ and radius $r>0$.
Now, we want to choose the function $f$ in the intensity measure of $\Pi$ so that at each time $t\in[0,T[$, the local intensity $f(t)$ of the process $\Pi$ on $\mathbb{R}^d$ compensates exactly the mean density of points that are not burned by time $t$.
In other words, we want the mean density of points that $\Pi$ puts in the unburned region $\left.\mathbb{R}^d\middle\backslash\mathcal{B}(t)\right.$ to be constantly equal to $1$.
To formalise this, let us fix $z\in\mathbb{R}^d$ and calculate the probability that $z$ is not burned by time $t$: we have
\[\begin{split}
\mathbb{P}(z\notin\mathcal{B}(t))&=\mathbb{P}\left(\Pi\left\{(x,s)\in\mathbb{R}^d\times[0,t]:\|x-z\|_1\leq\frac{t-s}{2}\right\}=0\right)\\
&=\exp\left(-\int_{\mathbb{R}^d}\int_0^t\mathbf{1}\left(\|x-z\|_1\leq\frac{t-s}{2}\right)\cdot f(s)\mathrm{d}s\mathrm{d}x\right)\\
&=\exp\left(-\int_0^t\frac{2^d}{d!}\cdot\left(\frac{t-s}{2}\right)^d\cdot f(s)\mathrm{d}s\right)\\
&=\exp\left(-\int_0^t\frac{\left(t-s\right)^d}{d!}\cdot f(s)\mathrm{d}s\right).
\end{split}\]
This probability does not depend on $z$, and corresponds to the mean density of points in $\mathbb{R}^d$ that are not burned by time $t$.
Thus, we want to choose $f$ so that
\[f(t)\cdot\exp\left(-\int_0^t\frac{\left(t-s\right)^d}{d!}\cdot f(s)\mathrm{d}s\right)=1\quad\text{for all $t\in[0,T[$.}\]
This certainly holds for $f=y^{(d)}$, if $y:[0,T[\rightarrow\mathbb{R}$ is a solution of \eqref{eq:cauchy}.
Indeed, by integrating $y^{(d+1)}=y\cdot y^{(d)}$ as a first order differential equation in $y^{(d)}$, we obtain that
\[y^{(d)}(t)=\exp\left(\int_0^ty(s)\mathrm{d}s\right) =\exp\left(\int_0^t \frac{\left(t-s\right)^d}{d!}\cdot y^{(d)}(s)\mathrm{d}s\right) \quad\text{for all $t\in[0,T[$,}\]
where the last equality follows by successive integration by parts.

Ultimately, we set the intensity measure of $\Pi$ to be $\mathrm{d}x\otimes y^{(d)}(t)\mathrm{d}t$ on $\mathbb{R}^d\times[0,T[$, where $y:{[0,T[}\rightarrow\mathbb{R}$ is the maximal solution of \eqref{eq:cauchy}.
For this continuous model of Poissonian burning, the blow-up rate of $y$ is relevant to determine whether the whole of $\mathbb{R}^d$ is burned eventually by the process.
More precisely, Theorem \ref{thm:main} has the following corollary (see Section \ref{sec:poisson} for more details).

\begin{cor}\label{cor:probabilistic}
Almost surely, the whole of $\mathbb{R}^d$ is burned eventually by the process, i.e, we have
\[\bigcup_{t\in[0,T[}\mathcal{B}(t)=\mathbb{R}^d.\]
\end{cor}


Finally, let us mention that we plan on interpreting this continuous model of Poissonian burning in $\mathbb{R}^d$ as the scaling limit of a discrete (in space and time) model of random burning in $\mathbb{Z}^d$ in a forthcoming paper \cite{BC25}.

\paragraph{Strategy of the proof of Theorem \ref{thm:main} and organisation of the paper.}
\begin{itemize}
\item In Section \ref{sec:preliminary}, we first prove that the maximal solution $y:{[0,T[}\rightarrow\mathbb{R}$ of \eqref{eq:cauchy} blows up in finite time, and provide rough bounds on $y$ and its derivatives near the explosion time $T$, relying on the absolute monotonicity of $y$,

\item In Section \ref{sec:reduction}, we consider some functions of $y$ and its derivatives that we know to remain bounded near the explosion time $T$, and observe that they satisfy a system of differential equations which, after a time change, is simplified into a Lotka--Volterra system,

\item In Section \ref{sec:lotkavolterra}, we study the long-time behaviour of solutions to this Lotka--Volterra system.
For $d\in\llbracket1,10\rrbracket$, a Lyapunov function argument shows that all solutions converge to the bulk stationary point, which implies \eqref{eq:goalstrong}.
Regardless of the value of $d$, the so-called method of average Lyapunov functions shows that the average in time of all solutions converges to the bulk stationary point, which implies \eqref{eq:goal}.
\end{itemize}

\paragraph{Phase transition in $d$.} As we will see in Section \ref{sec:lotkavolterra}, implementing the Lyapunov function argument mentioned above requires finding positive real numbers $\lambda_1,\ldots,\lambda_d$ for which the symmetric matrix
\begin{equation}\label{eq:onemillion}
\begin{bmatrix}
4\lambda_1&\lambda_2-\lambda_1&\lambda_3&\cdots&\lambda_d\\
\lambda_2-\lambda_1&2\lambda_2&-\lambda_2&&\\
\lambda_3&-\lambda_2&\ddots&\ddots&\\
\vdots&&\ddots&\ddots&-\lambda_{d-1}\\
\lambda_d& &&-\lambda_{d-1}&2\lambda_d
\end{bmatrix}
\end{equation}
is positive-definite.
At first, we found solutions $\lambda_1,\ldots,\lambda_d>0$ by hand for $d\in\llbracket1,9\rrbracket$, and naively thought that this would be possible for all values of $d$.
To our surprise, it turns out that such solutions $\lambda_1,\ldots,\lambda_d>0$ exist for and only for $d\in\llbracket1,10\rrbracket$, as can be proved using the Farkas lemma \cite{chris}.
Now, while this Lyapunov function argument fails for $d\geq11$, it could still be the case that all solutions of the Lotka--Volterra system converge to the bulk stationary point, which would imply \eqref{eq:goalstrong}.
However, numerical simulations suggest that for $d=11$, generic solutions of the Lotka--Volterra system oscillate around the bulk stationary point without converging (see Figure~\ref{fig:oscillation1} and Figure \ref{fig:oscillation2}).
Again, even without having the convergence of all solutions of the Lotka--Volterra system to the bulk stationary point, it could still be the case that \eqref{eq:goalstrong} holds (for instance, if the precise solution corresponding to the time-changed function of $y$ and its derivatives converges to the bulk stationary point).
It would be interesting to know what happens exactly, but this is beyond the scope of this paper.

\begin{figure}[ht]
\centering
\includegraphics[width=0.9\linewidth]{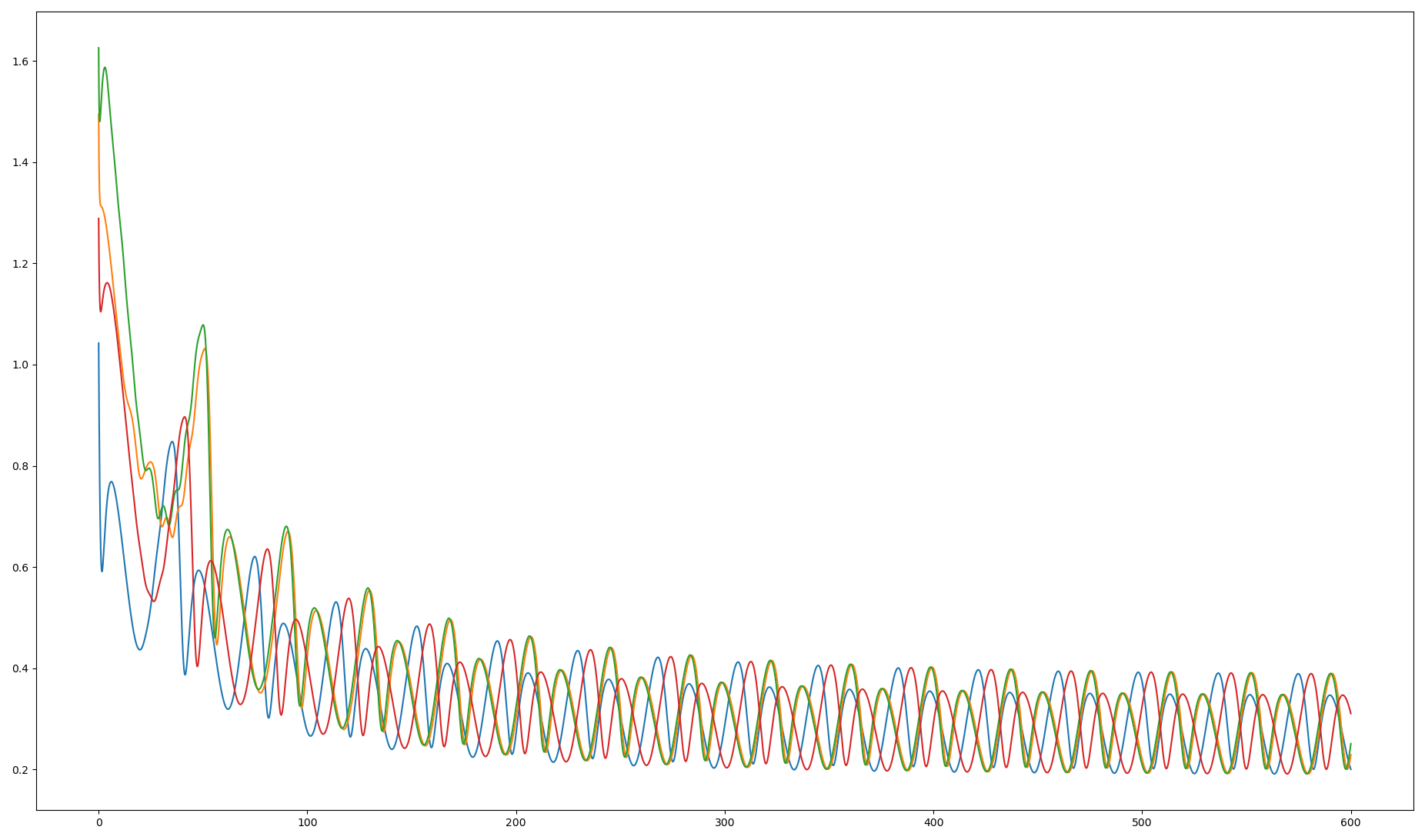}
\caption{Euclidean distance to the bulk stationary point for four generic solutions of the Lotka--Volterra system \eqref{eq:lotkavolterra} for $d=11$.}\label{fig:oscillation1}
\end{figure}

\begin{figure}[ht]
\centering
\includegraphics[width=0.9\linewidth]{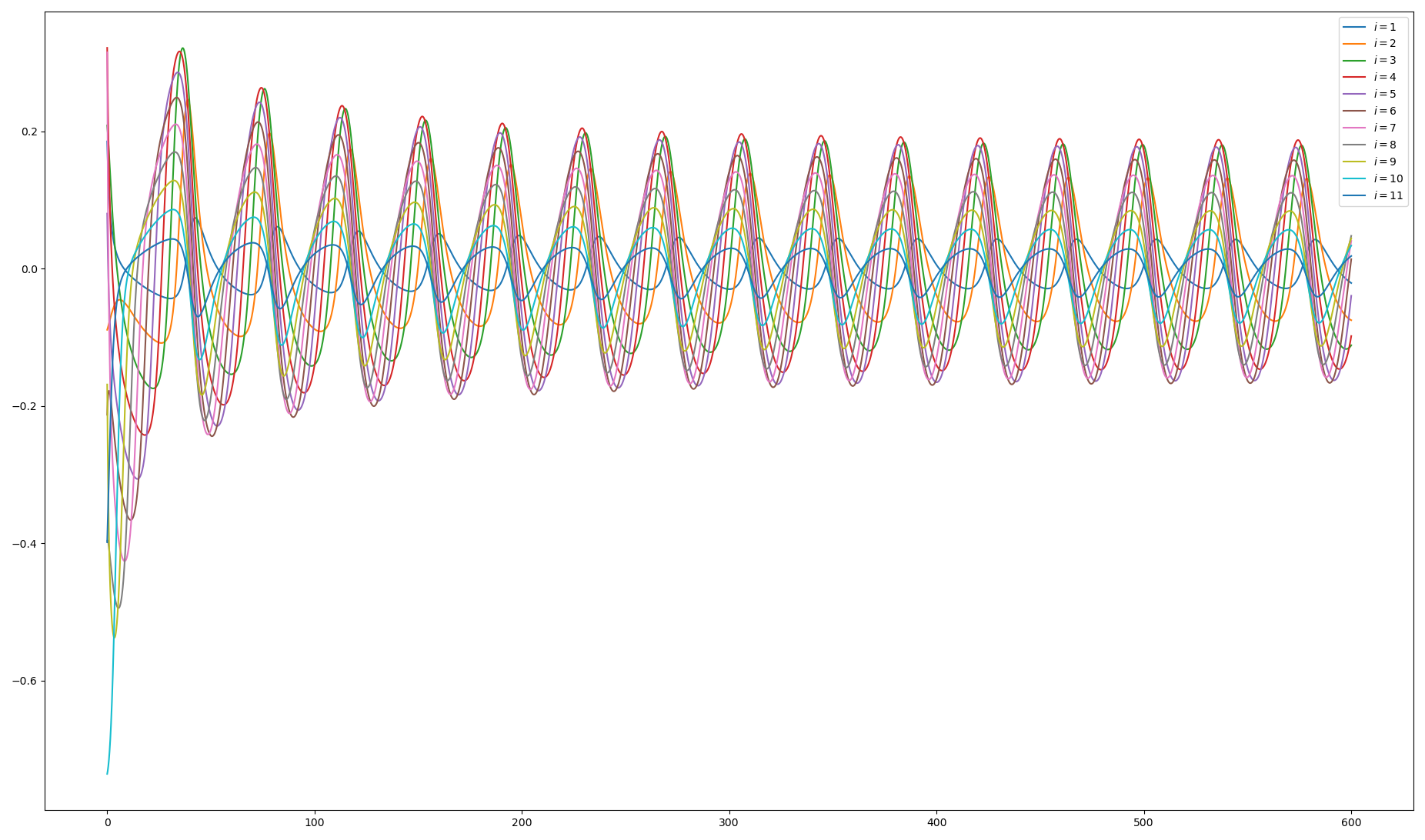}
\caption{Oscillation around the bulk stationary point for some generic solution of the Lotka--Volterra system \eqref{eq:lotkavolterra} for $d=11$, coordinate per coordinate.}\label{fig:oscillation2}
\end{figure}

\paragraph{Acknowledgements.}
We are extremely grateful to Patrick Gérard for suggesting the strategy of the proof of Theorem \ref{thm:main}, and for motivating discussions during the course of this project.
We warmly thank Christopher Criscitiello for helpful discussions around the feasibility of \eqref{eq:onemillion}.
We acknowledge support from ERC Consolidator grant ``SuPerGRandMa'', and benefitted from a PEPS JCJC grant.
G.B. is grateful to the members of LAGA at Université Sorbonne Paris Nord for their hospitality, and acknowledges support from SNSF Eccellenza grant 194648 and NCCR SwissMAP.

\section{Poissonian burning}\label{sec:poisson}

We start by explaining how the results of Theorem \ref{thm:main} allow us to study the probabilistic model of Poissonian burning introduced above.
Recall that $\Pi$ is a Poisson process with intensity ${\mathrm{d}x\otimes y^{(d)}(t)\mathrm{d}t}$ on $\mathbb{R}^d\times[0,T[$, where $y:{[0,T[}\rightarrow\mathbb{R}$ is the maximal solution of \eqref{eq:cauchy}.
We think of each atom $(x,t)$ of $\Pi$ as a source of fire appearing at position $x\in\mathbb{R}^d$ at time $t\in[0,T[$, the fire then propagating at speed $1/2$.
Equipping $\mathbb{R}^d$ with the $1$-norm $\|\cdot\|_1$, we let
\[\mathcal{B}(t)=\bigcup_{\substack{(x,s)\in\Pi\\s\leq t}}\overline{B}_{\|\cdot\|_1}\left(x,\frac{t-s}{2}\right)\]
be the set of points in $\mathbb{R}^d$ that are burned by time $t\in[0,T[$, and we let
\[\mathcal{B}=\bigcup_{t\in[0,T[}\mathcal{B}(t)=\bigcup_{(x,t)\in\Pi}B_{\|\cdot\|_1}\left(x,\frac{T-t}{2}\right)\]
be the set of points that are burned eventually, where $B_{\|\cdot\|_1}(x,r)=\left\{z\in\mathbb{R}^d:\|x-z\|_1<r\right\}$ denotes the $1$-norm open ball with centre $x\in\mathbb{R}^d$ and radius $r>0$.
Thanks to Theorem \ref{thm:main}, we are able to show that almost surely, the whole of $\mathbb{R}^d$ is burned eventually by the process, i.e, we have $\mathcal{B}=\mathbb{R}^d$ (Corollary \ref{cor:probabilistic}).

\begin{proof}[Proof of Corollary \ref{cor:probabilistic}]
The argument is similar to \cite[Proposition 2.5]{biermeestrade}, where Biermé and Estrade give a sufficient condition for ``high frequency coverage'' in the Mandelbrot--Shepp covering problem.
For each $\varepsilon>0$, notice that for all $z\in\mathbb{R}^d$, we have
\[\begin{split}
\mathbb{P}\left(\text{$B_{\|\cdot\|_1}(z,\varepsilon)$ is not contained in $\mathcal{B}$}\right)&\leq\mathbb{P}\left(z\notin\mathcal{B}(T-\varepsilon)\right)\\
&=\exp\left(-\int_0^{T-\varepsilon}\frac{(T-\varepsilon-s)^d}{d!}\cdot y^{(d)}(s)\mathrm{d}s\right)\\
&=\exp\left(-\int_0^{T-\varepsilon}y(t)\mathrm{d}t\right),
\end{split}\]
where the last equality follows by successive integration by parts.
Now, by \eqref{eq:goal}, we have
\begin{equation}\label{eq:goalweak}
\exp\left(-\int_0^{T-\varepsilon}y(t)\mathrm{d}t\right)=\varepsilon^{d+1+o(1)}\quad\text{as $\varepsilon\to0$.}
\end{equation}
We deduce that $\mathcal{B}=\mathbb{R}^d$ a.s, by the following standard argument.
Fix $R>0$, and let us show that $\left.\overline{B}_{\|\cdot\|_1}(0,R)\middle\backslash\mathcal{B}\right.=\emptyset$ a.s.
For each $\varepsilon>0$, let $\left(B_{\|\cdot\|_1}(z,\varepsilon)\,;\,z\in Z_\varepsilon\right)$ be a covering of $\overline{B}_{\|\cdot\|_1}(0,R)$ by balls of radius $\varepsilon$, with centres $z\in\overline{B}_{\|\cdot\|_1}(0,R)$ at least $\varepsilon$ apart from each other for the $1$-norm.
Since the balls $\left(B_{\|\cdot\|_1}(z,\varepsilon/2)\,;\,z\in Z_\varepsilon\right)$ are disjoint and included in $B_{\|\cdot\|_1}(0,R+\varepsilon)$, a measure argument entails that $\#Z_\varepsilon\leq(2R/\varepsilon+1)^d$.
Now, notice that we have the inclusion of events
\[\left(\left.\overline{B}_{\|\cdot\|_1}(0,R)\middle\backslash\mathcal{B}\right.\neq\emptyset\right)\subset\bigcup_{z\in Z_\varepsilon}\left(\text{$B_{\|\cdot\|_1}(z,\varepsilon)$ is not contained in $\mathcal{B}$}\right).\]
Thus, by the union bound, we get
\[\begin{split}
\mathbb{P}\left(\left.\overline{B}_{\|\cdot\|_1}(0,R)\middle\backslash\mathcal{B}\right.\neq\emptyset\right)&\leq\sum_{z\in Z_\varepsilon}\mathbb{P}\left(\text{$B_{\|\cdot\|_1}(z,\varepsilon)$ is not contained in $\mathcal{B}$}\right)\\
&\leq\left(\frac{2R}{\varepsilon}+1\right)^d\cdot\exp\left(-\int_0^{T-\varepsilon}y(t)\mathrm{d}t\right).
\end{split}\]
By \eqref{eq:goalweak}, the above right hand side goes to zero as $\varepsilon\to0$, which completes the proof.
\end{proof}

\section{Preliminary considerations}\label{sec:preliminary}

\begin{center}
\hrulefill \textit{ For the rest of the paper, we let $y:{[0,T[}\rightarrow\mathbb{R}$ be the maximal solution of \eqref{eq:cauchy}.} \hrulefill
\end{center}

We now move on to the proof of Theorem \ref{thm:main}.
In this section, we first establish some basic properties of $y$.

\subsection{Absolute monotonicity}

We start with the following ``absolute monotonicity'' property.

\begin{prop}\label{prop:absmon}
We have $y(t),\ldots,y^{(d-1)}(t)\geq0$ and $y^{(d)}(t)\geq1$ for all $t\in[0, T[$.
In particular, for each $i\in\llbracket0,d\rrbracket$, the function $y^{(i)}$ is strictly increasing on $[0, T[$, and we have $y^{(i)}(t)>0$ for all $t\in{]0, T[}$.
\end{prop}
\begin{proof}
Consider the set
\[O=\left\{t\in{[0,T[}:\text{$y(s),\dots,y^{(d-1)}(s)\geq0$ and $y^{(d)}(s)\geq1$ for all $s\in[0,t]$}\right\}.\]
For the first assertion, it suffices to prove that $O=[0,T[$.
First, observe that $O$ is a closed subset of $[0,T[$ that contains $0$.
Moreover, we claim that $O$ is also an open subset of $[0,T[$.
By the connectedness of $[0,T[$, this entails that $O=[0,T[$, as desired.
To prove the claim, fix $t\in O$.
Since $y^{(d)}(t)\geq1$, we may choose $\varepsilon>0$ small enough so that $y^{(d)}(s)\geq0$ for all $s\in[t,t+\varepsilon]$.
Now, let us show that $[0,t+\varepsilon]\subset O$.
By the definition of $O$, we automatically have $[0,t]\subset O$, so let us prove the complementary inclusion.
Since $y^{(d)}(s)\geq0$ for all $s\in[t,t+\varepsilon]$, the function $y^{(d-1)}$ is non-decreasing on $[t,t+\varepsilon]$, hence $y^{(d-1)}(s)\geq y^{(d-1)}(t)\geq0$ for all $s\in[t,t+\varepsilon]$.
In turn, the function $y^{(d-2)}$ is thus non-decreasing on $[t,t+\varepsilon]$, hence $y^{(d-2)}(s)\geq y^{(d-2)}(t)\geq0$ for all $s\in[t,t+\varepsilon]$, and so on.
The argument iterates up to $y(s)\geq y(t)\geq0$ for all $s\in[t,t+\varepsilon]$.
Finally, since $y^{(d+1)}(s)=y(s)\cdot y^{(d)}(s)\geq0$ for all $s\in[t,t+\varepsilon]$, the function $y^{(d)}$ is non-decreasing on $[t,t+\varepsilon]$, hence $y^{(d)}(s)\geq y^{(d)}(t)\geq1$ for all $s\in[t,t+\varepsilon]$.
This shows that $[t,t+\varepsilon]\subset O$, as we wanted to prove.

Now, let us prove the second assertion.
By the previous point, we have $y^{(d)}(t)\geq1$ for all $t\in[0,T[$, hence $y^{(d-1)}$ is strictly increasing on $[0,T[$.
Since $y^{(d-1)}(0)=0$, we deduce that $y^{(d-1)}(t)>0$ for all $t\in{]0,T[}$.
In turn, the function $y^{(d-2)}$ is thus strictly increasing on $[0,T[$, hence $y^{(d-2)}(t)>y^{(d-2)}(0)=0$ for all $t\in{]0,T[}$, and so on.
The argument iterates up to $y(t)>y(0)=0$ for all $t\in{]0,T[}$.
Finally, since $y^{(d+1)}(t)=y(t)\cdot y^{(d)}(t)>0$ for all $t\in{]0,T[}$, we get that $y^{(d)}$ is strictly increasing on $[0,T[$, which completes the proof.
\end{proof}

\subsection{Finite time blow-up}

We now focus on the blow-up of $y$.
Using the previous absolute monotonicity result, we start by proving that $y$ blows up in finite time.
Incidentally, this argument also provides the upper bound in \eqref{eq:goal}.
The lower bound, which is the one that we really need in view of Corollary \ref{cor:probabilistic}, requires more work.

\begin{prop}\label{prop:blowup}
The solution $y:{[0,T[}\rightarrow\mathbb{R}$ blows up in finite time: we have $T<\infty$, and $y(t)\rightarrow\infty$ as $t\to T^-$.
Moreover, we have
\[\varlimsup_{t\to T^-}\frac{1}{\ln(1/(T-t))}\cdot\int_0^ty(s)\mathrm{d}s\leq d+1.\]
\end{prop}
\begin{proof}
First, observe that $z:{[0,\tau[}\rightarrow\mathbb{R}$ is a (maximal) solution of \eqref{eq:cauchy} if and only if the function ${x:t\in{[0,\tau[}\mapsto\int_0^tz(s)\mathrm{d}s}$ is a (maximal) solution of the Cauchy problem
\begin{equation}\label{eq:cauchyx}
\begin{cases}
x^{(d+1)}(t)= \mathrm{e}^{x(t)},&t\in\mathbb{R}_+,\\
x(0),\ldots,x^{(d)}(0)=0.
\end{cases}
\end{equation}
Indeed, if $z:{[0,\tau[}\rightarrow\mathbb{R}$ is a solution of \eqref{eq:cauchy}, then by integrating $z^{(d+1)}=z\cdot z^{(d)}$ as a first order differential equation in $z^{(d)}$, we get that
\[z^{(d)}(t)=\exp\left(\int_0^tz(s)\mathrm{d}s\right)\quad\text{for all $t\in[0,\tau[$,}\]
hence the function $x:t\in{[0,\tau[}\mapsto\int_0^tz(s)\mathrm{d}s$ is a solution of \eqref{eq:cauchyx}.
Conversely, if $x:{[0,\tau[}\rightarrow\mathbb{R}$ is a solution of \eqref{eq:cauchyx}, then its derivative $z=x'$ satisfies $z^{(d)}= \mathrm{e}^x$, hence ${z^{(d+1)}=x'\cdot \mathrm{e}^x=z\cdot z^{(d)}}$.
Back to the proof of the proposition now, since $y:{[0,T[}\rightarrow\mathbb{R}$ is the maximal solution of \eqref{eq:cauchy}, we obtain that the function ${x:t\in{[0,T[}\mapsto\int_0^ty(s)\mathrm{d}s}$ is a maximal solution of \eqref{eq:cauchyx}.
Thus, by a standard corollary of the Cauchy--Lipschitz theorem, we have either $T=\infty$, or $T<\infty$ and $x(t)\rightarrow\infty$ as $t\to T^-$.
Ultimately, we want to show that the second case prevails, but before getting there, observe that even if $T=\infty$, we must have $x(t)\rightarrow\infty$ as $t\to T^-$.
Indeed, by Proposition \ref{prop:absmon}, the function $y$ is strictly increasing on $[0,T[$, and strictly positive on $]0,T[$.
In particular, if $T=\infty$, then we have
\[x(t)=\int_0^ty(s)\mathrm{d}s\geq\int_1^ty(1)\mathrm{d}s=(t-1)\cdot y(1)\underset{t\to\infty}{\longrightarrow}\infty.\]
Therefore, regardless of the finiteness of $T$, the function $x$ is an increasing bijection from $[0,T[$ to $\mathbb{R}_+$.
Now, consider the sequence $0=t_0<t_1<\ldots\in[0,T[$ defined by $x(t_n)=n$ for all $n\in\mathbb{N}$.
Since $T=\sum_{n\geq0}(t_{n+1}-t_n)$, it suffices to provide a summable bound on the $(t_{n+1}-t_n)$'s to prove the finiteness of $T$.
To this end, we argue as follows.
For each $n\in\mathbb{N}$, we have $x^{(d+1)}(t)= \mathrm{e}^{x(t)}\geq \mathrm{e}^n$ for all $t\in[t_n,t_{n+1}]$, hence
\[x^{(d)}(t)\geq(t-t_n)\cdot \mathrm{e}^n\quad\text{for all $t\in[t_n,t_{n+1}]$.}\]
We used here that $x^{(d)}(t_n)=y^{(d-1)}(t_n)\geq0$, by Proposition \ref{prop:absmon}.
Integrating the above inequality between $t_n$ and $t$ and iterating the argument, we obtain that
\[x^{(d-1)}(t)\geq\frac{(t-t_n)^2}{2}\cdot \mathrm{e}^n\quad\text{for all ${t\in[t_n,t_{n+1}]}$,}\]
and so on up to
\[x'(t)\geq\frac{(t-t_n)^d}{d!}\cdot \mathrm{e}^n\quad\text{for all $t\in[t_n,t_{n+1}]$.}\]
Integrating one last time from $t_n$ to $t_{n+1}$, we deduce that
\[1=x(t_{n+1})-x(t_n)\geq\frac{(t_{n+1}-t_n)^{d+1}}{(d+1)!}\cdot \mathrm{e}^n,\]
hence
\[t_{n+1}-t_n\leq\left((d+1)!\cdot \mathrm{e}^{-n}\right)^{1/(d+1)}=:C\cdot \mathrm{e}^{-c\cdot n},\]
with $c=1/(d+1)$ and $C=((d+1)!)^{1/(d+1)}$.
In particular, we obtain that $T=\sum_{n\geq0}(t_{n+1}-t_n)$ is finite, as desired.
Moreover, we obtain the bound
\[T-t_n=\sum_{k\geq n}(t_{k+1}-t_k)\leq\sum_{k\geq n}C\cdot \mathrm{e}^{-c\cdot k}=C\cdot\frac{ \mathrm{e}^{-c\cdot n}}{1- \mathrm{e}^{-c}}\]
for all $n\in\mathbb{N}$, which can be rearranged as
\[ \mathrm{e}^{c\cdot(n+1)}\leq\frac{C\cdot  \mathrm{e}^c}{1- \mathrm{e}^{-c}}\cdot\frac{1}{T-t_n}.\]
It follows that for all $t\in[t_n,t_{n+1}]$, we have
\[ \mathrm{e}^{c\cdot x(t)}\leq \mathrm{e}^{c\cdot(n+1)}\leq\frac{C\cdot \mathrm{e}^c}{1- \mathrm{e}^{-c}}\cdot\frac{1}{T-t_n}\leq\frac{C\cdot \mathrm{e}^c}{1- \mathrm{e}^{-c}}\cdot\frac{1}{T-t},\]
and since the furthermost sides of this bound do not depend on $n$, we conclude that
\[ \mathrm{e}^{c\cdot x(t)}\leq\frac{C\cdot \mathrm{e}^c}{1- \mathrm{e}^{-c}}\cdot\frac{1}{T-t}\quad\text{for all $t\in[0,T[$.}\]
Recalling that $c=1/(d+1)$, this yields
\[\varlimsup_{t\to T^-}\frac{x(t)}{\ln(1/(T-t))}\leq d+1,\]
as desired.
\end{proof}

\subsection{Rough bounds}

To prepare for the core of the proof, we conclude this section by establishing the following rough bounds on $y$ and its derivatives near the explosion time $T$.

\begin{prop}\label{prop:roughbounds}
For each $i\in\llbracket0,d\rrbracket$, we have $y^{(i)}(t)\rightarrow\infty$ as $t\to T^-$.
Furthermore, we have
\[\varliminf_{t\to T^-}(T-t)\cdot y(t)\geq2,\]
and for each $i\in\llbracket1,d\rrbracket$, we have
\[\varlimsup_{t\to T^-}\frac{y^{(i)}(t)}{y(t)\cdot y^{(i-1)}(t)}\leq1.\]
\end{prop}
\begin{proof}
Recall that by Proposition \ref{prop:blowup}, we have $T<\infty$ and $y(t)\rightarrow\infty$ as $t\to T^-$.
Now, fix $i\in\llbracket1,d\rrbracket$, and let us prove that $y^{(i)}(t)\rightarrow\infty$ as $t\to T^-$.
By Proposition \ref{prop:absmon}, the function $y^{(i)}$ is strictly increasing on $[0,T[$.
Moreover, this function cannot be bounded, for otherwise $y$, which can be expressed as a function of $y^{(i)}$ by integrating $i$ times, would be bounded since $T<\infty$.
This shows that $y^{(i)}(t)\rightarrow\infty$ as $t\to T^-$, as desired.

Now, we move on to the rough bounds on $y$ and its derivatives near the explosion time $T$; starting with the last one, for $i=d$.
First, since $y$ is a solution of \eqref{eq:cauchy}, we have
\[y^{(d)}(t)=1+\int_0^ty(s)\cdot y^{(d)}(s)\mathrm{d}s\quad\text{for all $t\in[0,T[$.}\]
Thus, by the absolute monotonicity of $y$ (Proposition \ref{prop:absmon}), we obtain the bound
\begin{equation}\label{eq:boundud}
y^{(d)}(t)\leq1+\int_0^ty(t)\cdot y^{(d)}(s)\mathrm{d}s=1+y(t)\cdot y^{(d-1)}(t)\quad\text{for all $t\in[0,T[$;}
\end{equation}
from which we deduce, since $y(t)\cdot y^{(d-1)}(t)\rightarrow\infty$ as $t\to T^-$, that
\[\varlimsup_{t\to T^-}\frac{y^{(d)}(t)}{y(t)\cdot y^{(d-1)}(t)}\leq1.\]
Next, integrating \eqref{eq:boundud} and iterating the argument, we obtain the bound
\[y^{(d-1)}(t)\leq t+\int_0^ty(s)\cdot y^{(d-1)}(s)\mathrm{d}s\leq t+\int_0^ty(t)\cdot y^{(d-1)}(s)\mathrm{d}s=t+y(t)\cdot y^{(d-2)}(t)\]
for all $t\in[0,T[$; from which we deduce, since $y(t)\cdot y^{(d-2)}(t)\rightarrow\infty$ as $t\to T^-$, that
\[\varlimsup_{t\to T^-}\frac{y^{(d-1)}(t)}{y(t)\cdot y^{(d-2)}(t)}\leq1.\]
This argument can be iterated up to obtaining the bound
\begin{equation}\label{eq:boundu1}
y'(t)\leq\frac{t^{d-1}}{(d-1)!}+\int_0^ty(s)\cdot y'(s)\mathrm{d}s\leq\frac{T^{d-1}}{(d-1)!}+\frac{y(t)^2}{2}\quad\text{for all $t\in[0,T[$;}
\end{equation}
from which we deduce, since $y(t)\rightarrow\infty$ as $t\to T^-$, that
\[\varlimsup_{t\to T^-}\frac{y'(t)}{y(t)^2}\leq\frac{1}{2}\leq1.\]
Finally, let us prove that $\varliminf_{t\to T^-}(T-t)\cdot y(t)\geq2$.
To ease notation, we set $c=\left.T^{d-1}\middle/(d-1)!\right.$.
Equation \eqref{eq:boundu1} may then be rearranged as
\[\frac{\left.y'(t)\middle/\sqrt{2c}\right.}{1+\left(y(t)\middle/\sqrt{2c}\right)^2}\leq\sqrt{\frac{c}{2}}\quad\text{for all $t\in[0,T[$,}\]
in which the left hand side is none other than the derivative of $t\in{[0,T[}\mapsto\arctan\left(y(t)\middle/\sqrt{2c}\right)$.
Since $y(t)\rightarrow\infty$ as $t\to T^-$, integrating the above equation from $t$ to $T$ yields
\[\frac{\pi}{2}-\arctan\left(\frac{y(t)}{\sqrt{2c}}\right)\leq\sqrt{\frac{c}{2}}\cdot(T-t)\quad\text{for all $t\in[0,T[$.}\]
Finally, since $\pi/2-\arctan(x)\sim1/x$ as $x\to\infty$, we deduce that
\[\frac{1+o(1)}{\left.y(t)\middle/\sqrt{2c}\right.}\leq\sqrt{\frac{c}{2}}\cdot(T-t)\quad\text{as $t\to T^-$,}\]
which can be rearranged as
\[(T-t)\cdot y(t)\geq2+o(1)\quad\text{as $t\to T^-$.}\]
This completes the proof.
\end{proof}

\section{Reduction to a Lotka--Volterra system}\label{sec:reduction}

In this section, we consider some functions of $y$ and its derivatives that we know to remain bounded near the explosion time~$T$ thanks to Proposition \ref{prop:roughbounds}, and observe that they satisfy a system of differential equations which, after a time change, is simplified into a Lotka--Volterra system.
For the sake of readability, we split this reduction argument into three straightforward claims (see Claim \ref{claim:1}, \ref{claim:2} and \ref{claim:3} below).

We then conclude this section by giving the proof of Theorem \ref{thm:main}, taking for granted the statement of Proposition \ref{prop:climax} below about the long-time behaviour of solutions to the Lotka--Volterra system mentioned above.
The latter result will be proved in the next section.

\paragraph{Time changes.}
Guided by the rough bounds obtained in Proposition \ref{prop:roughbounds}, we set
\[u_0(t)=\frac{1}{(T-t)\cdot y(t)}\quad\text{for all $t\in{]0,T[}$,}\]
and for each $i\in\llbracket1,d\rrbracket$, we let
\[u_i(t)=\frac{y^{(i)}(t)}{y(t)\cdot y^{(i-1)}(t)}\quad\text{for all $t\in{]0,T[}$.}\]
Note that by the absolute monotonicity of $y$ (Proposition \ref{prop:absmon}), these functions are well-defined and strictly positive.

\begin{claim}\label{claim:1}
The functions $u_0,u_1,\ldots,u_d$ satisfy the following system of differential equations: for all $t\in{]0,T[}$, we have
\begin{equation}\label{eq:systemu}
\begin{cases}
u_0'(t)=\frac{1}{T-t}\cdot(u_0(t)-u_1(t)),&\\
u_i'(t)=\frac{1}{T-t}\cdot\frac{u_i(t)(u_{i+1}(t)-u_1(t)-u_i(t))}{u_0(t)},&i\in\llbracket1,d-1\rrbracket,\\
u_d'(t)=\frac{1}{T-t}\cdot\frac{u_d(t)(1-u_1(t)-u_d(t))}{u_0(t)}.
\end{cases}
\end{equation}
\end{claim}
We spare the reader the details of these calculations, which are straightforward and use the fact that $y$ is a solution of \eqref{eq:cauchy} only for the last equation.
We then let $t_0=(1-1/ \mathrm{e}) T\in{]0,T[}$, and we introduce the time change
\[\begin{matrix}
\phi:&\mathbb{R}_+&\longrightarrow&[t_0,T[\\
&t&\longmapsto&T\cdot\left(1- \mathrm{e}^{-(1+t)}\right).
\end{matrix}\]
For each $i\in\llbracket0,d\rrbracket$, we set $v_i(t)=u_i(\phi(t))$ for all $t\in\mathbb{R}_+$.
Note that by the rough bounds of Proposition \ref{prop:roughbounds}, there exists a constant $C>0$ such that for each $i\in\llbracket0,d\rrbracket$, we have $v_i(t)\leq C$ for all $t\in\mathbb{R}_+$.

\begin{claim}\label{claim:2}
The functions $v_0,v_1,\ldots,v_d$ satisfy the following system of differential equations: for all $t\in\mathbb{R}_+$, we have
\begin{equation}\label{eq:systemv}
\begin{cases}
v_0'(t)=v_0(t)-v_1(t),&\\
v_i'(t)=\frac{v_i(t)(v_{i+1}(t)-v_1(t)-v_i(t))}{v_0(t)},&i\in\llbracket1,d-1\rrbracket,\\
v_d'(t)=\frac{v_d(t)(1-v_1(t)-v_d(t))}{v_0(t)}.&
\end{cases}
\end{equation}
Moreover, we have
\begin{equation}\label{eq:v0}
v_0(t)= \mathrm{e}^t\cdot\int_t^\infty v_1(s)\cdot \mathrm{e}^{-s}\mathrm{d}s\quad\text{for all $t\in\mathbb{R}_+$.}
\end{equation}
\end{claim}
\begin{proof}
Observe that $\phi$ is a smooth increasing bijection from $\mathbb{R}_+$ to $[t_0,T[$, and that
\[\phi'(t)=T\cdot \mathrm{e}^{-(1+t)}=T-\phi(t)\quad\text{for all $t\in\mathbb{R}_+$.}\]
Thus, for each $i\in\llbracket0,d\rrbracket$, we have
\[v_i'(t)=u_i'(\phi(t))\cdot\phi'(t)=u_i'(\phi(t))\cdot(T-\phi(t))\quad\text{for all $t\in\mathbb{R}_+$,}\]
and \eqref{eq:systemv} readily follows from \eqref{eq:systemu}.
Now, to prove \eqref{eq:v0}, let us set $w_0(t)=v_0(t)\cdot \mathrm{e}^{-t}$ for all $t\in\mathbb{R}_+$.
By \eqref{eq:systemv}, we have
\[w_0'(t)=v_0'(t)\cdot \mathrm{e}^{-t}-v_0(t)\cdot \mathrm{e}^{-t}=-v_1(t)\cdot \mathrm{e}^{-t}\quad\text{for all $t\in\mathbb{R}_+$.}\]
Moreover, since $v_0$ is bounded, we have $w_0(t)\rightarrow0$ as $t\to\infty$.
Thus, integrating the above equation between $t$ and infinity, we obtain
\[0-w_0(t)=-\int_t^\infty v_1(s)\cdot \mathrm{e}^{-s}\mathrm{d}s\quad\text{for all $t\in\mathbb{R}_+$,}\]
which yields \eqref{eq:v0}.
\end{proof}

Finally, we introduce one last time change $\psi$, defined as follows.
Observe that the function
\[t\in\mathbb{R}_+\longmapsto\int_0^t\frac{\mathrm{d}s}{v_0(s)}\]
is a smooth increasing bijection from $\mathbb{R}_+$ to $\mathbb{R}_+$ (since $v_0$ is strictly positive and bounded, the function $1/v_0$ is strictly positive and bounded away from zero), and let $\psi:\mathbb{R}_+\rightarrow\mathbb{R}_+$ be its inverse function.
For each $i\in\llbracket1,d\rrbracket$, we set
\begin{equation}\label{eq:w}
w_i(t)=v_i(\psi(t))\quad\text{for all $t\in\mathbb{R}_+$.}
\end{equation}

\begin{claim}\label{claim:3}
The functions $w_1,\ldots,w_d$ satisfy the following system of differential equations: for all $t\in\mathbb{R}_+$, we have
\begin{equation}\label{eq:systemw}
\begin{cases}
w_i'(t)=w_i(t)(w_{i+1}(t)-w_1(t)-w_i(t)),&i\in\llbracket1,d-1\rrbracket,\\
w_d'(t)=w_d(t)(1-w_1(t)-w_d(t)).&
\end{cases}
\end{equation}
\end{claim}
\begin{proof}
By construction, we have
\[\psi'(t)=v_0(\psi(t))\quad\text{for all $t\in\mathbb{R}_+$.}\]
Thus, for each $i\in\llbracket1,d\rrbracket$, we have
\[w_i'(t)=v_i'(\psi(t))\cdot\psi'(t)=v_i'(\psi(t))\cdot v_0(\psi(t))\quad\text{for all $t\in\mathbb{R}_+$,}\]
and \eqref{eq:systemw} readily follows from \eqref{eq:systemv}.
\end{proof}

We recognise in \eqref{eq:systemw} a Lotka--Volterra system, from which it is reasonable to hope obtaining useful information on $w_1,\ldots,w_d$.
Namely, we obtain the following information.

\begin{prop}\label{prop:climax}
\begin{itemize}
\item For $d\in\llbracket1,10\rrbracket$, the following holds: for each $i\in\llbracket1,d\rrbracket$, we have
\[w_i (t)\longrightarrow\frac{i}{d+1}\quad\text{as $t\to\infty$.}\]

\item Regardless of the value of $d$, there exists a constant $c>0$ such that for each $i\in\llbracket1,d\rrbracket$, we have $w_i(t)\geq c$ for all $t\in\mathbb{R}_+$.
Moreover, for each $i\in\llbracket1,d\rrbracket$, we have
\[\frac{1}{t}\cdot\int_0^tw_i(s)\mathrm{d}s\longrightarrow\frac{i}{d+1}\quad\text{as $t\to\infty$.}\]
\end{itemize}
\end{prop}

We postpone the proof of Proposition \ref{prop:climax} to the next section.

\paragraph{Back to $y$.}
We conclude this section by explaining how the above information on $w_1,\ldots,w_d$ can be traced back to $y$ and its derivatives to yield the results of Theorem \ref{thm:main}.

\begin{proof}[Proof of Theorem \ref{thm:main} assuming Proposition \ref{prop:climax}]
\begin{itemize}
\item Let $ d \in \llbracket1,10\rrbracket$.
By Proposition \ref{prop:climax}, for each $i\in\llbracket1,d\rrbracket$, we have $v_i(t)=w_i\left(\psi^{-1}(t)\right)\rightarrow i/(d+1)$ as $t\to\infty$.
In particular, by \eqref{eq:v0}, we deduce that $v_0(t)\rightarrow1/(d+1)$ as $t\to\infty$.
Now, for each $i\in\llbracket0,d\rrbracket$, we get that $u_i(t)=v_i\left(\phi^{-1}(t)\right)\rightarrow(1\vee i)/(d+1)$ as $t\to T^-$.
In particular, we obtain that
\[(T-t)\cdot y(t)=\frac{1}{u_0(t)}\longrightarrow d+1\quad\text{as $t\to T^-$.}\]
Moreover, observe that for each $j\in\llbracket1,d\rrbracket$, we have
\[\prod_{i=1}^ju_i(t)=\frac{y^{(j)}(t)}{y(t)^{j+1}}\quad\text{for all $t\in[t_0,T[$,}\]
hence
\[(T-t)^{j+1}\cdot y^{(j)}(t)=((T-t)\cdot y(t))^{j+1}\cdot\prod_{i=1}^ju_i(t)\underset{t\to T^-}{\longrightarrow}(d+1)^{j+1}\cdot\prod_{i=1}^j\frac{i}{d+1}=(d+1)\cdot j!.\]

\item Let $d\in\mathbb{N}^*$. By Proposition \ref{prop:climax}, for each $i\in\llbracket1,d\rrbracket$, we have $v_i(t)=w_i\left(\psi^{-1}(t)\right)\geq c$ for all $t\in\mathbb{R}_+$.
In particular, by \eqref{eq:v0}, we deduce that $v_0(t)\geq c$ for all $t\in\mathbb{R}_+$.
Now, for each $i\in\llbracket0,d\rrbracket$, we get that $u_i(t)=v_i\left(\phi^{-1}(t)\right)\geq c$ for all $t\in[t_0,T[$.
Moreover, recall that by Proposition \ref{prop:roughbounds}, there exists a constant $C>0$ such that $u_i(t)=v_i\left(\phi^{-1}(t)\right)\leq C$ for all $t\in[t_0,T[$.
In particular, we obtain that
\[(T-t)\cdot y(t)=\frac{1}{u_0(t)}\in\left[\frac{1}{C},\frac{1}{c}\right]\quad\text{for all $t\in[t_0,T[$.}\]
Furthermore, for each $j\in\llbracket1,d\rrbracket$, we get that
\[(T-t)^{j+1}\cdot y^{(j)}(t)=((T-t)\cdot y(t))^{j+1}\cdot\prod_{i=1}^ju_i(t)\in\left[\frac{c^j}{C^{j+1}},\frac{C^j}{c^{j+1}}\right]\quad\text{for all $t\in[t_0,T[$.}\]
This completes the proof of the first assertion.

Now, let us prove the second assertion; namely, that
\[\frac{1}{\ln(1/(T-t))}\cdot\int_0^ty(s)\mathrm{d}s\longrightarrow d+1\quad\text{as $t\to T^-$.}\]
Since $\ln(1/(T-\phi(t)))\sim t$ as $t\to\infty$, this is equivalent to showing that
\[\frac{1}{t}\cdot\int_{\phi(0)}^{\phi(t)}y(s)\mathrm{d}s\longrightarrow d+1\quad\text{as $t\to\infty$.}\]
Now, for each $t\in\mathbb{R}_+$, we have
\[\int_{\phi(0)}^{\phi(t)}y(s)\mathrm{d}s=\int_{\phi(0)}^{\phi(t)}\frac{\mathrm{d}s}{(T-s)\cdot u_0(s)}=\int_0^t\frac{\phi'(r)\mathrm{d}r}{(T-\phi(r))\cdot u_0(\phi(r))}=\int_0^t\frac{\mathrm{d}r}{v_0(r)}=\psi^{-1}(t).\]
So we want to prove that $\left.\psi^{-1}(t)\middle/t\right.\rightarrow d+1$ as $t\to\infty$, or equivalently that
\begin{equation}\label{eq:meanchi}
\frac{\psi(t)}{t}\longrightarrow\frac{1}{d+1}\quad\text{as $t\to\infty$.}
\end{equation}
To prove \eqref{eq:meanchi}, observe that on the one hand, since $\psi'=v_0(\psi(\cdot))$, we have
\[\int_0^tw_1(s)\mathrm{d}s=\int_0^t\frac{v_1(\psi(s))}{v_0(\psi(s))}\cdot\psi'(s)\mathrm{d}s=\int_0^{\psi(t)}\frac{v_1(r)}{v_0(r)}\mathrm{d}r\quad\text{for all $t\in\mathbb{R}_+$,}\]
so that by assumption, we have
\[\frac{1}{t}\cdot\int_0^{\psi(t)}\frac{v_1(s)}{v_0(s)}\mathrm{d}s=\frac{1}{t}\cdot\int_0^tw_1(s)\mathrm{d}s\longrightarrow\frac{1}{d+1}\quad\text{as $t\to\infty$.}\]
On the other hand, by integrating $v_0'/v_0=(v_0-v_1)/v_0$ from $0$ to $\psi(t)$, we obtain
\[\ln(v_0(\psi(t)))-\ln(v_0(0))=\psi(t)-\int_0^{\psi(t)}\frac{v_1(s)}{v_0(s)}\mathrm{d}s\quad\text{for all $t\in\mathbb{R}_+$,}\]
which can be rearranged as
\[\frac{\psi(t)}{t}=\frac{1}{t}\cdot\int_0^{\psi(t)}\frac{v_1(s)}{v_0(s)}\mathrm{d}s+\frac{1}{t}\cdot\ln\left(\frac{v_0(\psi(t))}{v_0(0)}\right)\quad\text{for all $t\in\mathbb{R}_+^*$.}\]
Finally, since $v_0$ is bounded away from zero and infinity, we have
\[\frac{1}{t}\cdot\ln\left(\frac{v_0(\psi(t))}{v_0(0)}\right)\longrightarrow0\quad\text{as $t\to\infty$,}\]
and we deduce that \eqref{eq:meanchi} holds, which completes the proof.
\end{itemize}
\end{proof}

\section{Study of the Lotka--Volterra system}\label{sec:lotkavolterra}

In this section, we study the long-time behaviour of solutions to the Lotka--Volterra system that appears in Claim \ref{claim:3}:
\begin{equation}\label{eq:lotkavolterra}
\begin{cases}
w_i'(t)=w_i(t)(w_{i+1}(t)-w_1(t)-w_i(t)),&i\in\llbracket1,d-1\rrbracket,\\
w_d'(t)=w_d(t)(1-w_1(t)-w_d(t)),&
\end{cases}\quad t\in\mathbb{R}_+,
\end{equation}
with initial condition $w(0)\in\left(\mathbb{R}_+^*\right)^d$.
The main result of this section is stated in Proposition \ref{prop:lessbasic} below.

Before diving into the study of \eqref{eq:lotkavolterra}, let us emphasise the following point: in the previous section, the letters $w_1,\ldots,w_d$ referred to the precise functions defined in \eqref{eq:w}.
By Claim \ref{claim:3}, these functions satisfy \eqref{eq:lotkavolterra}, however with a specific initial condition in $\left(\mathbb{R}_+^*\right)^d$ on which we do not have explicit information.
Therefore, in this section, we study the long-time behaviour of all solutions to the Lotka--Volterra system \eqref{eq:lotkavolterra} with initial condition in $\left(\mathbb{R}_+^*\right)^d$.
From now on, we will use the letter $w=(w_1,\ldots,w_d)$ as a variable for the system \eqref{eq:lotkavolterra}, and to denote its solutions.
The system \eqref{eq:lotkavolterra} may be rewritten in a concise way as
\[w'(t)=F(w(t)),\quad t\in\mathbb{R}_+,\]
 where ${F=(F_1,\ldots,F_d):\mathbb{R}^d\rightarrow\mathbb{R}^d}$ is defined by:
\begin{equation}\label{eq:F}
\begin{cases}
F_i(w)=w_i(w_{i+1}-w_1-w_i),&i\in\llbracket1,d-1\rrbracket,\\
F_d(w)=w_d(1-w_1-w_d),&
\end{cases}\quad\text{for all $w=(w_1,\ldots,w_d)\in\mathbb{R}^d$.}
\end{equation}
Contrary to the previous section where we considered a specific solution of \eqref{eq:lotkavolterra}, we do not know a priori that every maximal solution of \eqref{eq:lotkavolterra} with initial condition in $\left(\mathbb{R}_+^*\right)^d$ is global, and has strictly positive and bounded coordinates: our first task is to prove that this is indeed the case.

\begin{prop}\label{prop:basic}
Fix $w^0=\left(w^0_1,\ldots,w^0_d\right)\in\left(\mathbb{R}_+^*\right)^d$, and let $w=(w_1,\ldots,w_d):{[0,\tau[}\rightarrow\mathbb{R}^d$ be the maximal solution of \eqref{eq:lotkavolterra} with initial condition $w(0)=w^0$.
This solution has strictly positive and bounded coordinates, and is global: we have $\tau=\infty$, and there exists a constant $C>0$ such that for each $i\in\llbracket1,d\rrbracket$, we have $0<w_i(t)\leq C$ for all $t\in\mathbb{R}_+$.
\end{prop}
\begin{proof}
The fact that $\left(\mathbb{R}_+^*\right)^d$ is stable by \eqref{eq:lotkavolterra} is a standard result on Lotka--Volterra systems (see, e.g, \cite[Theorem 6]{baigent}).
Now, fix $w^0=\left(w^0_1,\ldots,w^0_d\right)\in\left(\mathbb{R}_+^*\right)^d$, and let
\[w=(w_1,\ldots,w_d):{[0,\tau[}\longrightarrow\left(\mathbb{R}_+^*\right)^d\]
be the maximal solution of \eqref{eq:lotkavolterra} with initial condition $w(0)=w^0$, which we know to have strictly positive coordinates by the previous point.
We claim that the function
\[M:t\in{[0,\tau[}\longmapsto w_1(t)\vee\ldots\vee w_d(t)\vee1\]
is non-increasing.
Indeed, fix $t_0\in[0,\tau[$, and consider the set
\[O=\left\{t\in{[t_0,\tau[}:\text{$M(s)\leq M(t_0)$ for all $s\in[t_0,t]$}\right\}.\]
To prove the claim, it suffices to show that $O=[t_0,\tau[$.
First, observe that $O$ is a closed subset of $[t_0,\tau[$ that contains $t_0$.
Moreover, we claim that $O$ is also an open subset of $[t_0,\tau[$.
By the connectedness of $[t_0,\tau[$, this entails that $O=[t_0,\tau[$, as desired.
To prove that $O$ is open, fix $t\in O$, and let us show that there exists $\varepsilon>0$ such that $[t_0,t+\varepsilon]\subset O$.
By the definition of $O$, we automatically have $[t_0,t]\subset O$, so let us prove the complementary inclusion.
We set ${I=\left\{i\in\llbracket1,d\rrbracket:w_i(t)=M(t)\right\}}$.
On the one hand, for each $i\in\llbracket1,d\rrbracket\setminus I$, we have $w_i(t)<M(t)$ by definition, so there exists $\varepsilon_i>0$ such that $w_i(s)\leq M(t)$ for all $s\in[t,t+\varepsilon_i]$.
On the other hand, for each $i\in I$, we have
\[w_i'(t)\leq w_i(t)(M(t)-w_1(t)-w_i(t))=w_i(t)\cdot(-w_1(t))<0,\]
so there exists $\varepsilon_i>0$ such that $w_i'(s)\leq0$ for all $s\in[t,t+\varepsilon_i]$.
Thus, the function $w_i$ is non-increasing on $[t,t+\varepsilon_i]$, and we get that $w_i(s)\leq w_i(t)=M(t)$ for all $s\in[t,t+\varepsilon_i]$.
Setting ${\varepsilon=\varepsilon_1\wedge\ldots\wedge\varepsilon_d}$, we conclude that $w_i(s)\leq M(t)\leq M(t_0)$ for all $s\in[t,t+\varepsilon]$, which shows that $[t,t+\varepsilon]\subset O$, as we wanted to prove.
In particular, we obtain that for each $i\in\llbracket1,d\rrbracket$, we have $0<w_i(t)\leq M(0)$ for all $t\in[0,\tau[$.
Finally, since $w$ is maximal, the fact that it is bounded entails that it is global, i.e, that $\tau=\infty$, which completes the proof of the proposition.
\end{proof}

Before coming to the main result of this section, let us introduce some additional notation that will be convenient for the proof.
We set
\[A=\begin{bmatrix}
2&-1&&\\
1&1&\ddots&\\
\vdots&&\ddots&-1\\
1&&&1
\end{bmatrix}\quad\text{and}\quad b=\begin{bmatrix}
0\\
\vdots\\
0\\
1
\end{bmatrix},\]
so that \eqref{eq:lotkavolterra} may be rewritten as
\[\begin{cases}
w_i'(t)=w_i(t)\left(b_i-\sum_{j=1}^da_{i,j}w_j(t)\right),&i\in\llbracket1,d\rrbracket,
\end{cases}\quad t\in\mathbb{R}_+.\]
Observe that the matrix $A$ invertible.
Indeed, if $w\in\mathbb{R}^d$ is such that $Aw=0$, then we have
\[\begin{cases}
2w_1-w_2=0,\\
w_1+w_2-w_3=0,\\
\ldots,\\
w_1+w_{d-1}-w_d=0,\\
w_1+w_d=0,\\
\end{cases}\quad\text{i.e,}\quad\begin{cases}
w_2=2w_1,\\
w_3=w_1+w_2=3w_1,\\
\ldots,\\
w_d=w_1+w_{d-1}=dw_1,\\
(d+1)w_1=0,\\
\end{cases}\]
hence $w=0$.
Moreover, similar manipulations show that
\[A^{-1}\cdot b=\left(\frac{1}{d+1},\ldots,\frac{d}{d+1}\right)=:w^*.\]

The following proposition settles the long-time behaviour of solutions to \eqref{eq:lotkavolterra}.
In particular, it implies Proposition \ref{prop:climax}, which yields the results of Theorem \ref{thm:main}.

\begin{prop}\label{prop:lessbasic}
\begin{itemize}
\item For $d\in\llbracket1,10\rrbracket$, the following holds: for every solution $w:\mathbb{R}_+\rightarrow\left(\mathbb{R}_+^*\right)^d$ of \eqref{eq:lotkavolterra}, we have $w(t)\rightarrow w^*$ as $t\to\infty$.

\item Regardless of the value of $d$, the following holds.
For every solution $w:\mathbb{R}_+\rightarrow\left(\mathbb{R}_+^*\right)^d$ of \eqref{eq:lotkavolterra}, there exists a constant $c>0$ such that for each $i\in\llbracket1,d\rrbracket$, we have $w_i(t)\geq c$ for all $t\in\mathbb{R}_+$.
Moreover, for each $i\in\llbracket1,d\rrbracket$, we have
\[\frac{1}{t}\cdot\int_0^tw_i(s)\mathrm{d}s\longrightarrow\frac{i}{d+1}\quad\text{as $t\to\infty$.}\]
\end{itemize}
\end{prop}
\begin{proof}
\begin{itemize}
\item By a theorem of Goh \cite[Theorem 12]{baigent} (see \cite[Theorem 1]{goh} for the original result), it suffices to show that for ${d\in\llbracket1,10\rrbracket}$, there exists $\lambda_1,\dots,\lambda_d>0$ such that the symmetric matrix
\begin{equation}\label{eq:onemillion'}
DA+(DA)^*=\begin{bmatrix}
4\lambda_1&\lambda_2-\lambda_1&\lambda_3&\cdots&\lambda_d\\
\lambda_2-\lambda_1&2\lambda_2&-\lambda_2&&\\
\lambda_3&-\lambda_2&\ddots&\ddots&\\
\vdots&&\ddots&\ddots&-\lambda_{d-1}\\
\lambda_d& &&-\lambda_{d-1}&2\lambda_d
\end{bmatrix}
\end{equation}
is positive-definite, where $D$ is the diagonal matrix with diagonal entries $\lambda_1,\ldots,\lambda_d$, and $(DA)^*$ denotes the transpose of the matrix $DA$.
The idea of Goh's theorem is that the positive-definiteness of $DA+(DA)^*$ entails that
\[V:w\in\left(\mathbb{R}_+^*\right)^d\longmapsto\sum_{i=1}^d\lambda_i\cdot\left(w_i-w_i^*-w_i^*\cdot\ln\left(\frac{w_i}{w_i^*}\right)\right)\]
is a so-called Lyapunov function, in the sense that it has the following properties:
\begin{itemize}
\item We have $V(w)\geq0$ for all $w\in\left(\mathbb{R}_+^*\right)^d$, with equality if and only if $w=w^*$,

\item For every solution $w:\mathbb{R}_+\rightarrow\left(\mathbb{R}_+^*\right)^d$ of \eqref{eq:lotkavolterra}, we have
\[\frac{\mathrm{d}[V(w(t))]}{\mathrm{d}t}=-\left\langle w(t)-w^*,\frac{DA+(DA)^*}{2}\cdot(w(t)-w^*)\right\rangle\quad\text{for all $t\in\mathbb{R}_+$,}\]
and by the positive-definiteness of $DA+(DA)^*$, we have
\[-\left\langle w-w^*,\frac{DA+(DA)^*}{2}\cdot(w-w^*)\right\rangle\leq0\quad\text{for all $w\in\left(\mathbb{R}_+^*\right)^d$,}\]
with equality if and only if $w=w^*$.
\end{itemize}
Thus, by the LaSalle invariance principle (see, e.g, \cite[Theorem 10]{baigent}), we get that for every solution $w:\mathbb{R}_+\rightarrow\left(\mathbb{R}_+^*\right)^d$ of \eqref{eq:lotkavolterra}, we have $w(t)\rightarrow w^*$ as $t\to\infty$, as desired.

Now, it remains to show that for $d\in\llbracket1,10\rrbracket$, there exists ${\lambda_1,\ldots,\lambda_d>0}$ such that the symmetric matrix in \eqref{eq:onemillion'} is positive definite.
By the Sylvester criterion, it suffices to exhibit $\lambda_1,\ldots,\lambda_{10}>0$ such that
\[\Delta_k:=\begin{vmatrix}
4\lambda_1&\lambda_2-\lambda_1&\lambda_3&\cdots&\lambda_k\\
\lambda_2-\lambda_1&2\lambda_2&-\lambda_2&&\\
\lambda_3&-\lambda_2&\ddots&\ddots&\\
\vdots&&\ddots&\ddots&-\lambda_{k-1}\\
\lambda_k& &&-\lambda_{k-1}&2\lambda_k
\end{vmatrix}>0\quad\text{for all $k\in\llbracket1,10\rrbracket$.}\]
For $(\lambda_1,\ldots,\lambda_{10})=(1024,227,118,92,89,97,116,153,232,481)$, we find that
\begin{align*}
\Delta_1&=4096,\\
\Delta_2&=1224375,\\
\Delta_3&=114265104,\\
\Delta_4&=6270814340,\\
\Delta_5&=280372975336,\\
\Delta_6&=12330415584972,\\
\Delta_7&=687248010753336,\\
\Delta_8&=69483419810465760,\\
\Delta_9&=12807765625815100744,\\
\Delta_{10}&=136953089422286895648,
\end{align*}
which completes the proof of the first point.

\item First, let us prove that for every solution $w:\mathbb{R}_+\rightarrow\left(\mathbb{R}_+^*\right)^d$ of \eqref{eq:lotkavolterra}, there exists a constant $c>0$ such that for each $i\in\llbracket1,d\rrbracket$, we have $w_i(t)\geq c$ for all $t\in\mathbb{R}_+$.
For Lotka--Volterra systems, this property is called \emph{permanence}.
By a theorem of Jansen \cite[Lemma 2]{kirlinger} (see \cite[Theorem 3]{jansen} for the original result, see also \cite{hofbauersigmund}), it suffices to show that the function
\[\Psi:w\in\mathbb{R}_+^d\longmapsto\sum_{i=1}^d\left(b_i-\sum_{j=1}^da_{i,j}w_j\right)=1-(d+1)w_1\]
has the following property: for each stationary point $w\in\partial\left(\mathbb{R}_+^d\right)=\left.\mathbb{R}_+^d\middle\backslash\left(\mathbb{R}_+^*\right)^d\right.$ of \eqref{eq:lotkavolterra}, i.e, for each point $w\in\partial\left(\mathbb{R}_+^d\right)$ such that $F(w)=0$, where $F$ is defined in \eqref{eq:F}, we have $\Psi(w)>0$.
The idea of Jansen's theorem is that the assumption on $\Psi$ entails that $P:w\in\mathbb{R}_+^d\mapsto w_1\cdot\ldots\cdot w_d$ is a so-called average Lyapunov function, in the sense that it has the following properties:
\begin{itemize}
\item We have $P(w)\geq0$ for all $w\in\mathbb{R}_+^d$, with equality if and only if $w\in\partial\left(\mathbb{R}_+^d\right)$,

\item For every solution $w:\mathbb{R}_+\rightarrow\mathbb{R}_+^d$ of \eqref{eq:lotkavolterra}, we have
\[\frac{\mathrm{d}[P(w(t))]}{\mathrm{d}t}=\Psi(w(t))\cdot P(w(t))\quad\text{for all $t\in\mathbb{R}_+$,}\]
and the assumption on $\Psi$ entails that if $w(0)\in\partial\left(\mathbb{R}_+^d\right)$, then there exists $\tau>0$ such that
\[\int_0^\tau\Psi(w(t))\mathrm{d}t>0.\]
\end{itemize}
Thus, by Hofbauer's method of average Lyapunov functions \cite[Lemma 1]{kirlinger} (see \cite[Theorem 1]{hofbauer} for the original result), we get the permanence of \eqref{eq:lotkavolterra}, as desired.

To check the assumption on $\Psi$, let $w\in\mathbb{R}_+^d$ be such that ${F(w)=0}$, and suppose that $w_1>0$. Then, since $F_1(w)=0$, we must have $w_2-2w_1=0$, hence $w_2=2w_1>0$.
In turn, since $F_2(w)=0$, we must have $w_3-w_1-w_2=0$, hence $w_3=3w_1>0$, etc.
The argument iterates up to $w_d=dw_1>0$; and finally, since $F_d(w)=0$, we must have $1-w_1-w_d=0$, hence $(d+1)w_1=1$.
This shows that $w=w^*\in\left(\mathbb{R}_+^*\right)^d$. In particular, if $w\in\partial\left(\mathbb{R}_+^d\right)$, then we must have $w_1=0$, hence ${\Psi(w)=1>0}$, as desired.


Now, let us prove the second assertion; namely, that for every solution $w:\mathbb{R}_+\rightarrow\left(\mathbb{R}_+^*\right)^d$ of \eqref{eq:lotkavolterra}, we have
\[\frac{1}{t}\cdot\int_0^tw(s)\mathrm{d}s\longrightarrow w^*\quad\text{as $t\to\infty$.}\]
For each $i\in\llbracket1,d\rrbracket$, integrating $w_i'/w_i=b_i-\sum_{j=1}^da_{i,j}w_j$ from $0$ to $t$, we obtain
\[\ln w_i(t)-\ln w_i(0)=b_i\cdot t-\sum_{j=1}^da_{i,j}\cdot\int_0^tw_j(s)\mathrm{d}s\quad\text{for all $t\in\mathbb{R}_+$,}\]
which can be rearranged as
\[\sum_{j=1}^da_{i,j}\cdot\frac{1}{t}\cdot\int_0^tw_i(s)\mathrm{d}s=b_i-\frac{1}{t}\cdot\ln\left(\frac{w_i(t)}{w_i(0)}\right)\quad\text{for all $t\in\mathbb{R}_+^*$.}\]
Setting, for each $i\in\llbracket1,d\rrbracket$,
\[\overline{w_i}(t)=\frac{1}{t}\cdot\int_0^tw_i(s)\mathrm{d}s\quad\text{and}\quad\varepsilon_i(t)=\frac{1}{t}\cdot\ln\left(\frac{w_i(t)}{w_i(0)}\right)\quad\text{for all $t\in\mathbb{R}_+^*$,}\]
we may rewrite this in matrix notation as $A\cdot\overline{w}(t)=b-\varepsilon(t)$, i.e, $\overline{w}(t)=A^{-1}\cdot(b-\varepsilon(t))$.
Now, since for each $i\in\llbracket1,d\rrbracket$, the function $w_i$ is bounded away from zero and infinity, we get that $\varepsilon_i(t)\rightarrow0$ as $t\to\infty$.
It follows that $\overline{w}(t)\rightarrow A^{-1}\cdot b=w^*$, as desired.
\end{itemize}
\end{proof}

\bibliographystyle{siam}
\bibliography{biblio.bib}

\end{document}